\documentclass[11pt]{nyjm}
\usepackage{amsmath,amsthm,amssymb,hyperref,fancyhdr,mathtools}
\usepackage{derivative}
\usepackage[thinc]{esdiff}
\usepackage{enumitem}
\usepackage{doi}

\theoremstyle{plain}
\newtheorem{theorem}{Theorem}
\newtheorem{lemma}[theorem]{Lemma}
\newtheorem{proposition}[theorem]{Proposition}

\newtheorem{example}[theorem]{Example}

\numberwithin{theorem}{section}
\numberwithin{equation}{section}

\newcommand{\norm}[1]{\ensuremath{\left\|#1\right\|}}
\newcommand{\ip}[2]{\ensuremath{\left\langle#1,#2\right\rangle}}

\newcommand{\R}{\mathbb{R}}
\newcommand{\C}{\mathbb{C}}
\newcommand{\N}{\mathbb{N}}

\DeclareMathOperator{\spann}{span}

\newcommand{\fL}{\mathfrak{L}}

\newcommand{\cA}{\mathcal{A}}

\newcommand{\cS}{\mathcal{S}}
\newcommand{\cF}{\mathcal{F}}

\newcommand{\vm}{\textbf{\textit{m}}}

\newcommand{\cbuop}{C_{b,u}(\F^2)}
\newcommand{\bddf}{L^\infty(\C^n)}

\newcommand{\F}{\cF^2(\C^n)}

\newcommand{\Lpf}{L^p(\C^n, dm)}
\newcommand{\Lonef}{L^1(\C^n,dm)}
\newcommand{\LtwoF}{L^2(\C^n,d\lambda)}

\newcommand{\czero}{\mathcal{C}_0(\mathbb{C}^n)}
\newcommand{\symb}{\cA^2_\lambda}
\newcommand{\dom}{\mathfrak{D}}

\newcommand{\bdd}{\fL(\cF^2)}
\newcommand{\toepalg}{\mathfrak{T}(L^\infty)}
\newcommand{\traceop}{\cS^1(\cF^2)}
\newcommand{\comp}{\mathcal{K}}
\newcommand{\schop}{\cS^p(\cF^2)}

\newcommand{\tconv}[1]{\varphi_t \ast #1}

\newcommand{\actop}[2]{L_{#1}(#2)}
\newcommand{\actf}[2]{\ell_{#1}(#2)}

\newcommand{\dl}[1]{\ dm{(#1)}}

\newcommand{\Ftwo}{{2,\lambda}}
\newcommand{\Fp}{{p,\lambda}}

\newcommand{\Lone}{{1}}

\newcommand{\Tr}{\mathrm{Tr}}

\keywords{Toeplitz operators, Fock space, heat flow, heat semigroup, Berger-Coburn, compactness, Schatten class}
\title[Berger-Coburn theorem]{A quantum harmonic analysis approach to the Berger-Coburn theorem}

\author{Vishwa Dewage }
\author{Mishko Mitkovski }
\thanks{M.~M. was supported in part by NSF grant DMS-2000236.}

\address{Department of Mathematics, Clemson University\\
South Carolina, SC 29634, USA\\
E-Mail Dewage:vdewage@clemson.edu
E-Mail Mitkovski: mmitkov@clemson.edu}
\subjclass[2000]{22D25, 30H20, 47B35, 47L80}

\begin{document}


\begin{abstract}
    We use quantum harmonic analysis for densely defined operators to provide a simplified proof of the Berger-Coburn theorem for boundedness of Toeplitz operators. In addition, we revisit compactness and Schatten-class membership of densely defined Toeplitz operators.
\end{abstract}

\maketitle



\section{Introduction}

Werner's quantum harmonic analysis (QHA) \cite{W84} has now been established as an effective tool in several fields of mathematics, including time-frequency analysis, mathematical physics, and operator theory \cite{BBLS22, DM23, DDMO24, F19, FH23, FR23, KLSW12, LS20}. As noted in \cite{F19}, QHA is particularly well-suited to studying Toeplitz operators on the Bargmann-Fock space.
In this article, we use QHA methods to provide a simple and concise proof of the Berger-Coburn boundedness theorem \cite{BC94}, along with proofs of several subsequent results by Bauer, Coburn, and Isralowitz \cite{BCI10}. 

The Bargmann-Fock space $\F$ is the closed subspace of $L^2(\C^n, d\lambda)$ consisting of all entire functions that are square integrable with respect to the Gaussian measure $d\lambda$, $d\lambda(z)=e^{-\pi|z|^2}dm(z)$ ($dm$ denotes the Lebesgue measure on $\C^n\simeq \R^{2n}$). The orthogonal projection $P:L^2(\C,d\lambda)\to \cF^2(\C^n)$ is given by
\[P f(z)=\int_{\C^n} f(w)e^{\pi z\Bar{w}}\ d\lambda(w).\]  Given a function $a:\C^n\to \C$, a Toeplitz operator  $T_a:\F \to \F$ with symbol $a$ is formally defined by $$T_a f(z)=P(af)(z) \ \ \forall \ z\in \C^n.$$ $T_a$ is bounded if its symbol $a$ is a bounded function. However, it is well-known that there exist bounded Toeplitz operators with unbounded symbols. Berger and Coburn~\cite{BC94} tried to characterize unbounded symbols which give rise to bounded Toeplitz operators in terms of "the heat flow" of their symbol (also known as the $t$-Berezin transform) given by
\[B_t(a):=a\ast \varphi_{t}, \ \ t>0.\] 
More precisely, they proved the following now well-known result. 

\begin{theorem} [Berger-Coburn]\label{theo:BC_old}
Let $a\in \symb:=\{a:\C^n\to \C \mid a(\cdot - z)\in \LtwoF \ \forall z\in \C^n\}.$ If $T_a$ is a bounded operator then $B_t(a)$ is a bounded function for all $t>1/2$, and
\[\norm{B_t(a)}_{\infty}\leq C_1(t)\norm{T_a}\]
for some constant $C_1(t)>0$.
On the other hand, if $B_t(a)$ is a bounded function for some $0<t<1/2$, then $T_a$ is a bounded operator  and 
\[\norm{T_a}\leq C_2(t) \norm{B_t(a)}_{\infty}\]
for some constant $C_2(t)>0$.
\end{theorem}

Berger and Coburn were hopeful that the only absent value of $t$ in the theorem above would yield the desired characterization of boundedness in the form: $T_a$ is bounded if and only of $B_{1/2}(a)=a\ast \varphi_{1/2}$ is bounded. This still remains an open question. However, the necessity and sufficiency of this condition were established for special classes of Toeplitz operators in \cite{CHS19,CHS21,CHS24} and \cite{X24}   It was shown later by Bauer, Coburn, and Isralowitz \cite{BCI10} that this same heat transform can be used to almost characterize compactness, as well as the Schatten-class membership of densely defined Toeplitz operators. In both papers~\cite{BC94, BCI10}, the more challenging part of the proofs relied heavily on certain results about pseudo-differential operators. While these results are fairly standard, their application felt somewhat artificial. In a recent paper, Fulsche and Bauer~\cite{BF20} presented a direct proof of one of the Berger-Coburn inequalities, entirely bypassing the need for pseudo-differential operators.

The main goal of this note  is to provide a simple, concise proofs of the results in~\cite{BC94, BCI10} using methods from quantum harmonic analysis. Our proofs completely avoid pseudo-differential operators and are extremely simple and short, showing the elementary nature of these estimates. In addition, we answer one of the questions raised in~\cite{BCI10} by providing a quantitative analog of Theorem ~\ref{theo:BC_old} for Schatten-class operators.


\section{Preliminaries} \label{sec:prelim}

It is well-known that the Bargmann-Fock space $\F$ is a reproducing kernel Hilbert space with reproducing kernel
$K_z(w)=e^{\pi\Bar{z}w}$ and normalized reproducing kernel $$k_z(w)=\frac{K_z(w)}{\|K_z\|}=e^{\pi\Bar{z}w-\frac{\pi}{2}|z|^2},$$
where $\Bar{z}w=\Bar{z}_1w_1+\cdots+\Bar{z_n}w_n$. The monomials $$e_{\vm}(z)=\sqrt{\frac{\pi^\vm}{\vm!}}z^{\vm},\ \ \ \vm\in\N_0^n$$
form an orthonormal basis for $\F$.



\subsection{Heat kernel} The \textit{heat kernel} 
\[\varphi_t(z)=\frac{1}{t^n}e^{-\frac{\pi|z|^2}{t}}, \ \ t>0,\] plays a crucial role in our proofs. It is well-known that it is an approximate identity, i.e., for $a\in L^p(\C^n), p\in [1,\infty)$ we have $\tconv{a}\to a$ in $L^p$ and similarly for $a$ uniformly continuous, we have $\tconv{a}\to a$ uniformly.   

\subsection{Quantum harmonic analysis (QHA)}
Here, we recall some basic QHA concepts and properties that we will use in our proofs. The proofs of these properties can be found in~\cite{F19, H23, LS18, W84}. We denote by $\bdd$ the algebra of all bounded operators on the Fock space. For $1\leq p<\infty$, let $\schop$ denote the space of all Schatten-$p$ class operators on $\F$.




The QHA translation of an operator $S\in \bdd$ by $z\in \C^n$ is defined by
$$\actop{z}{S}=W_zS W_z^*, \ \ S\in \bdd,$$ where $$W_zf(w)=k_z(w)f(w-z);\ \ w\in \C^n, \ f\in \F,$$ are the Weyl unitary operators acting on $\F$.
For $S\in \bdd$ the map $z\to \actop{z}{S}$ is continuous w.r.t. the strong operator topology (SOT). We say an operator $S\in \bdd$ is uniformly continuous if the map $z\mapsto \actop{z}{S},\ \C^n\to \bdd$ is continuous. We denote the set of all bounded uniformly continuous operators by $\cbuop$. It is a $C^*$-subalgebra of $\bdd$ containing all compact operators.

 The \emph{convolution} of a function $\psi:\C^n\to \C$ and an operator $S\in \bdd$ is defined formally by
$$S\ast \psi :=\psi\ast S:=\int   \actop{z} S \ \psi(z) \dl{z}.$$ Here, and everywhere below, the operator-valued integrals are defined in the weak (Pettis) sense, i.e., as unique operators satisfying 
\[\ip{\int   \actop{z} S \ \psi(z)\dl{z} f}{g}=\int\ip{   \actop{z} S f }{g}  \psi(z)\dl{z},\] for every $f, g\in \F$.

In the following cases, the convolution $\psi\ast S$ is a well-defined bounded operator, and in addition satisfies the QHA Young's inequalities: For $1\leq p\leq \infty$ we have,
\begin{enumerate}[label=(\roman*)]
    \item If $\psi\in L^1(\C^n, dm)$ and $S\in \cS^p(\cF^2)$ then $\psi\ast S\in \cS^p(\cF^2)$ and\\
    $\|\psi\ast S\|_p\leq \|\psi\|_1\|S\|_p$.
    \item If $\psi\in L^p(\C^n,dm)$ and $S\in \cS^1(\cF^2)$ then  $\psi\ast S\in \cS^p(\cF^2)$ and \\
    $\|\psi\ast S\|_p\leq \|\psi\|_1\|S\|_p$.
    \item  In addition, $\psi \ast S\in \cbuop$ in both of the above cases.
\end{enumerate}

We also have the following basic properties of QHA translations and convolutions for the above cases where the convolution is well-defined. Let $\psi,\psi_1,\psi_2:\C^n\to \C$ and $S\in\bdd$ s.t. the following convolutions are defined. Then the translation commutes with convolution and the convolution is associative:
    \begin{enumerate}[label=(\roman*)]
        \item $\actop{z}{\psi\ast S}=(\actf{z} {\psi})\ast S =\psi\ast (\actop{z}{S}) $.
        \item $\psi_1\ast (\psi_2\ast S)=(\psi_1\ast \psi_2)\ast S=(\psi_2\ast \psi_1)\ast S$.
    \end{enumerate} 


The QHA convolution between two operators $S\in \bdd$ and $T\in \traceop$, denoted $S\ast T$, is defined to be the function on $\C^n$ given by
\[S\ast T(z):=\Tr(SL_z(U_0TU_0))\]
where $U_0$ is the parity operator on $\F$ given by
$$(U_0f)(z)=f(-z),\ \ \ z\in \C^n,\ f\in \F.$$
Then $T\ast S$ is uniformly continuous and $\|T\ast S\|_\infty\leq \|T\|_1\|S\|$.

Furthermore, if $1\leq p\leq \infty$ and $S\in  \cS^p(\cF^2)$, then $S\ast T\in L^p(\C^n, dm)$ and satisfies the following QHA Young's inequality:
$$\|S\ast T\|_p\leq \|S\|_p\|T\|_1.$$

For $A,C\in\bdd$, $B\in \traceop$, the following identities hold
    \begin{enumerate}[label=(\roman*)]
        \item $A\ast B= B\ast A$.
        \item $(A\ast B)\ast C=A\ast (B\ast C)$ if one of $A$ and $C$ is trace-class.
    \end{enumerate}  
 Also if $A,B\in \traceop$ and $a\in \bddf$, we have
$$a\ast (A\ast B)=(a\ast A)\ast B.$$

\subsection{Toeplitz operators as convolutions} The primary link between QHA and Toeplitz operators is the observation that Toeplitz operators can be represented as convolutions, i.e for $a\in \bddf$, we have
\[T_a=a\ast P_0,\] where $P_0:=1\otimes 1$, where $1$ denotes the constant function $1$. 
  Then from the basic properties of convolutions, we get: for $a\in L^\infty(\C^n)$ and $\psi\in L^1(\C^n)$, $\actop{z}{T_a}=T_{\actf{z} a}$ and  $\psi\ast T_a= T_{\psi \ast a}$.

As convolutions $a\ast P_0$ are in $\cbuop$, the set of all Toeplitz operators with bounded symbols, and thus the whole Toeplitz algebra $\toepalg$, which is the $C^*$-algebra generated by the set of all Toeplitz operators with bounded symbols, are both contained in $\cbuop$. In ~\cite{F19}, Fulsche proved that these two $C^*$-algebras coincide.

\subsection{Berezin transform as a convolution} The Berezin transform of a function $a$ is defined as 
\[B(a)(z)=\int a(w)|\ip{k_z}{k_w}|^2dm(w). \] In other words $B(a)(z)=\ip{T_a k_z}{k_z}$. The Berezin transform of an operator $S\in \bdd$, is defined as
$$B(S)(z)=\ip{Sk_z}{k_z}, \ z\in\C^n.$$
Clearly, $B(T_a)=B(a)$
when the integral exists for all $z\in \C^n$

It is easy to see that 
\[B(a)=a \ast \varphi_1, \text{ and } B(S)=S\ast P_0.\] 

The following proposition is an easy consequence of the fact that the convergence in strong operator topology (SOT) implies the convergence in weak operator topology (WOT).
\begin{proposition}\label{prop:SOT_convergence}
    Let $\{S_k\}$ be a sequence in $\bdd$ s.t. $S_k\to S\in \bdd$ in the SOT. Then $B(S_k)\to B(S)$ pointwise.
\end{proposition}

\section{Operator heat semigroup}
We now introduce an operator heat semigroup $\{\Phi_t\}_{t>0}$, that behaves analogous to the function heat semigroup $\{\varphi_t\}_{t>0}$.

We let $U:=2^nU_0$ where $U_0$ is the parity operator, and define 
$$\Phi_t=\varphi_t\ast U, \quad t>0.$$
Let $\mathcal{P}_k(\C^n)$ be the space of homogeneous polynomials of order $k$, and let $P_k$ be the projection onto $\mathcal{P}_k(\C^n)$, given by:
$$P_k=\sum_{|\vm|=k} E_\vm,$$
where $E_{\vm}=e_\vm \otimes e_\vm$.

\begin{lemma}
    The Berezin transform of $U$ is given by
    $$U\ast P_0=\varphi_{1/2}.$$
\end{lemma}
\begin{proof}
    Now we note that the Berezin transform of $U$ is the Gaussian $\varphi_{1/2}$: for $z\in \C^n$,
\begin{align*}
    (U\ast P_0)(z)&=\ip{Uk_z}{k_z}\\
    &=2^ne^{-\pi|z|^2}\ip{U_0K_z}{K_z}\\
    &= 2^ne^{-\pi|z|^2}K_z(-z)\\
    &=2^ne^{-2\pi|z|^2}\\
    &=\varphi_{1/2}(z).
\end{align*}
\end{proof}
\begin{proposition}\label{lem:Phi_t_initial}
For $t>0$, $\Phi_t$ is a radial operator and is given by \[\Phi_t=\frac{2^n}{(2t+1)^n} \sum_{k=0}^\infty \Big(\frac{2t-1}{2t+1}\Big)^k P_k,\]
where the series converges in the strong operator topology.
\end{proposition}
\begin{proof}
It is easy to check that the series on the right converges in the SOT and defines a bounded operator.
The fact that $\Phi_t$ is radial follows from the fact that  the Berezin transform of $\Phi_t$ is a radial function, i.e.,
\begin{align*}
    \Phi_t\ast P_0=(\varphi_t\ast U)\ast P_0=\varphi_t\ast (U\ast P_0)=\varphi_t\ast \varphi_{1/2}= \varphi_{t+1/2}.
\end{align*}
 So to show the equality it is enough to check that their Berezin transforms are equal.
 We first observe that for a multi-index $\vm\in \N_0^n$, 
 \[(E_\vm\ast P_0)(z)=\Tr(E_\vm\actop{z}{P_0})=\ip{k_z}{e_\vm}\Tr(e_\vm\otimes \overline{k_z})=\pi^{|\vm|}\frac{z^\vm\overline{z}^\vm}{\vm!}e^{-\pi|z|^2},\]
 for $z\in \C^n$.
 Then for $k\in \N_0$,
 \begin{align*}
     \sum_{|\vm|=k} (E_\vm\ast P_0)(z) = \frac{\pi^k|z|^{2k}}{k!}e^{-\pi|z|^2}
 \end{align*}
 Now by Proposition \ref{prop:SOT_convergence}, 
 \begin{align*}
     \bigg( \frac{2^n}{(2t+1)^n} &\sum_{k=0}^\infty \Big(\frac{2t-1}{2t+1}\Big)^k \sum_{|\vm|=k} E_\vm\bigg)\ast P_0\\ 
     &= \frac{2^n}{(2t+1)^n} \sum_{k=0}^\infty \Big(\frac{2t-1}{2t+1}\Big)^k \sum_{|\vm|=k} (E_\vm\ast P_0)(z)\\
     &=\frac{1}{(t+1/2)^n}e^{-\frac{\pi |z|^2}{ t+1/2}}\\
     &=\varphi_{t+1/2}(z)\\
     &= (\Phi_t\ast P_0)(z).
 \end{align*}
 This completes the proof. 
\end{proof}

The operator heat equation \[ \Delta T_t=\pi \frac{\partial}{\partial t} T_t, \ \ T_0=U,\] is solved by the operator heat semigroup $\{\Phi_t\}_{t> 0}$ (see \cite{DM24}, Proposition 4.2). Even though strictly speaking $\{\Phi_t\}_{t> 0}$ does not form a semigroup, we still refer to it as such due to the following properties of $\{\Phi_t\}_{t> 0}$, which follows from the semigroup property of $\{\varphi_t\}_{t>0}$:
\begin{proposition}\label{prop:Phi_t} 
The following properties hold
\begin{enumerate}
    \item $\Phi_t=\varphi_t\ast \Phi=T_{\varphi_t}$ for $t>1/2$ and $\Phi_{1/2}=P_0$.
    \item $\varphi_t\ast \Phi_s=\Phi_{t+s}$ for $s,t>0$.
    \item $\Phi_s\ast \Phi_t=\varphi_{s+t+1}$ for $s,t>0$.
    \item $\Phi_t$ is a radial trace-class operator whose trace-norm is given by
    $$\|\Phi_t\|_1=\begin{array}{cc}
  \Bigg\{ 
    \begin{array}{cc}
     1,& \ t\geq 1/2 \\
      \cfrac{1}{(2t)^n},& \ t\in (0,1/2)  
    \end{array}
\end{array}$$
\end{enumerate}
\end{proposition}
\begin{proof}
     (1) follows by comparing Berezin transforms computed in the proof of Lemma \ref{lem:Phi_t_initial}.\\
     Proof of (2): By the associativity of convolutions and the semigroup property of the heat semigroup, we have that
     \begin{align*}
         \varphi_t\ast \Phi_s= \varphi_t\ast \varphi_s\ast U=\varphi_{t+s}\ast U= \Phi_{t+s}.
     \end{align*}
     Proof of (3): We prove that the Berezin transforms of the two functions coincide. Note that
     \begin{align*}
         (\Phi_s\ast \Phi_t)\ast \varphi_1&= (\Phi_s\ast \Phi_t)\ast (\Phi\ast \Phi)\\
         &= (\Phi_t\ast \Phi)\ast (\Phi_s\ast \Phi)\\
         &=\varphi_{t+1/2}\ast \varphi_{s+1/2}\\
         &=\varphi_{s+t}\ast \varphi_1.
     \end{align*}
     Then the result follows by the injectivity of the Berezin transform.\\
     Proof of (4): To compute the trace-norm, note that for $\vm\in \N_0^n$,
 \begin{align*}
     \ip{\Phi_t e_\vm}{e_\vm}&= \frac{2^n}{(2t+1)^n} \Big(\frac{2t-1}{2t+1}\Big)^{|\vm|}\ip{ E_\vm e_\vm}{e_\vm}\\
     &= \frac{2^n(2t-1)^{|\vm|}}{(2t+1)^{n+|\vm|}}.
 \end{align*}
Thus we have
\begin{align*}
    \|\Phi_t\|_1&=\sum_{\vm\in \N_0^n} |\ip{\Phi_t e_\vm}{e_\vm}|\\
    &=\frac{2^n}{(2t+1)^n}  \sum_{k=0}^\infty \Big|\frac{2t-1}{2t+1}\Big|^{k} \Big(\sum_{|\vm|=k}1\Big)\\
    &= \frac{2^n}{(2t+1)^n}  \sum_{k=0}^\infty \Big|\frac{2t-1}{2t+1}\Big|^{k} {k+n-1\choose k}\\
    &=\frac{2^n}{(2t+1)^n}  \Big(1-\Big|\frac{2t-1}{2t+1}\Big|\Big)^{-n}.
\end{align*}
\end{proof}


\section{Densely defined operators and quantum harmonic analysis}

We now define QHA convolutions for some unbounded functions and some densely defined operators. Toeplitz operators with unbounded symbols are instances of such convolutions.

\subsection{A class of symbols}
In \cite{BC94}, the authors point out that the condition $aK_z\in \LtwoF$ for all $z\in \C^n$ is sufficient for the Toeplitz operator $T_a$ to be densely defined. It is easy to verify that this is equivalent to the condition $\actf{z}{a}\in \LtwoF$ for all $z\in \C^n$. We denote by $\symb$ the set
$$\symb=\{a:\C^n\to \C \mid \actf{z}{a}\in \LtwoF \ \forall z\in \C^n\}.$$
Clearly, $\bddf\subset \symb$.

The following proposition discusses the behavior of convolutions of functions in $\symb$ with heat kernels.

\begin{proposition}\label{prop:function_convol}
    Let $a\in \symb$. Then 
    \begin{enumerate}
        \item The convolutions $\varphi_t\ast a$ are defined for all $t>0$ and $\varphi_t\ast a\in \symb$.
        \item $\varphi_t\ast a=a\ast \varphi_t,$ for all $t>0$.
        \item For $s,t>0$, we have 
        $$\varphi_s\ast (\varphi_t\ast a)=(\varphi_s\ast \varphi_t)\ast a=\varphi_{s+t}\ast a.$$
        \item Let $1\leq p\leq \infty$. If $\varphi_{t_0}\ast a\in \Lpf$ for some $t_0>0$, then $\varphi_t\ast a\in \Lpf$ for all $t>t_0$.
    \end{enumerate}
\end{proposition}

\begin{proof}
    Proof of (1): First, note that $\varphi_tK_z\in \Lonef$ for all $z\in \C^n$. This is true by H\"older's inequality as $K_z$ is square integrable w.r.t. the probability measure $\varphi_t(z)dz$ with $L^2$-norm equal to $e^{\frac{1}{2}\pi t|z|^2}$.
    Then $z\in \C^n$,
    \begin{align*}
        \|(\varphi_t\ast |a|)K_z\|_\Fp&=\Bigg(\int_{\C^n}\bigg(\int_{\C^n} |\varphi_t(x)a(w-x) K_z(w)| \ dx \bigg)^2  \ d\lambda(w) \Bigg)^{1/2}\\
        &\leq  \int_{\C^n} |\varphi_t(x)| \bigg(\int_{\C^n}| a(w-x) K_z(w)|^2  \ d\lambda(w) \bigg)^{1/2} dx \\
        &\hspace{5.4cm}\text{(by Minkowski inequality)}\\   
        &= \int_{\C^n} |\varphi_t(x)| \bigg(\int_{\C^n}| a(w) K_z(w+x)|^2  \ d\lambda(w) \bigg)^{1/2} dx \\
        &\hspace{5.6cm}\text{(by a change of variable)}\\
        &= \int_{\C^n} |\varphi_t(x)K_z(x)| \bigg(\int_{\C^n}| a(w) K_z(w)|^2  \ d\lambda(w) \bigg)^{1/2} dx \\
        &\hspace{4.55cm}\text{(as $K_z(w+x)=K_z(w)K_z(x)$)}\\
        &= \|aK_z\|_\Ftwo\|\varphi_t K_z\|_\Lone.     
    \end{align*}
Therefore, $(\varphi_t\ast |a|)K_z$ and $(\varphi_t\ast a)K_z$ are measurable functions that are in $\Lpf$.
Thus, $\varphi_t\ast a\in \symb$.

Proof of (2): This follows by a simple change of variable.

Proof of (3): Note that by Tonelli's theorem,
$$\varphi_{s}\ast (\varphi_t\ast |a|)= (\varphi_{s}\ast \varphi_t) \ast |a|=\varphi_{s+t} \ast |a|.$$
Moreover, by (1) $\varphi_{s+t} \ast |a|$ is finite almost everywhere.
Therefore, the result follows from Fubini's theorem, by replacing $|a|$ by $a$ in the above computation.

Proof of (4): Follows from (3) and Young's inequality.
\end{proof}

\subsection{Convolutions of densely defined operators}
We denote by $\dom$ the linear subspace
$$\dom:=\spann\{k_z\mid z\in \C^n\}.$$
Note that $\dom$ is a set satisfying the invariance
$\pi(\C^n)\dom=\dom$ because for
$z,w\in \C^n$, $$\pi(w)k_z=\pi(w)\pi(z)1=e^{i\mathrm{Im}(w\overline{z})}\pi(w+z)1=e^{i\mathrm{Im}(w\overline{z})}k_{w+z}.$$
Let $\psi:\C^n\to \C$ be a measurable function and let $S$ be a densely defined operator on $\dom$. 
Assume that for all $f\in \dom$, there exists $C_f>0$ s.t. 
$$\int_{\C^n} |\psi(z)||\ip{\actop{z}{S}f}{g}| \dl{z} \leq C_{f}\|g\|_2,\ \ \forall g\in \F.$$
Then we define $\psi\ast S$ to be the densely defined operator on $\dom$ defined weakly by
$$\psi\ast S=\int_{\C^n} \psi(z) \actop{z}{S} \dl{z}.$$

\begin{proposition}\label{prop:Toeplitz_convo}
    Let $a\in \symb$. Then $a\ast \Phi$ is a densely defined operator on 
    $\dom$ and $T_a=a\ast \Phi$ on $\dom$.
\end{proposition}

\begin{proof}
    It is enough to verify that $\psi \ast S$ is defined on $\{K_z\mid z\in \C^n\}$. Note that for $z\in\C^n$ and $f\in \F$, 
    \begin{align*}
        \int_{\C^n} |a(w)  \ip{\actop{w}{\Phi}K_z}{f}&| \dl{w}
        = \int_{\C^n} |a(w) \ip{(k_w\otimes k_w)k_z}{f}|  \dl{w}\\
       &= \int_{\C^n} |a(w) \ip{K_z}{K_w} \ip{K_w}{f}|  \ d\lambda(w)\\
       &= \int_{\C^n} |a(w) K_z(w) \overline{f(w)}|  \ d\lambda(w)\\
       &\leq \|aK_z\|_\Ftwo \|f\|_2,
    \end{align*}
    where the last line follows by Cauchy-Schwarz inequality.
    Hence $\psi \ast \Phi$ is defined on $\{K_z\mid z\in \C^n\}$ and by following the same steps as above we have
    $$\ip{(a\ast \Phi)K_z}{f}=\int_{\C^n} a(w) K_z(w)\overline{f(w)} \ d\lambda(w)=\ip{T_a K_z}{f},$$
    proving $a\ast \Phi=T_a$.
\end{proof}

We have the following canonical example that was discussed in \cite{BC94}.
\begin{example}
    Let $\xi \in \C\setminus \{0\}$ s.t. $\frac{\mathrm{Re}\  \xi}{|\xi|^2}>-1/2$. Consider the natural extension $\varphi_\xi$ of the heat kernel defined by 
    $$\varphi_\xi (z)=\frac{1}{\xi^n}e^{- \frac{\pi|z|^2}{\xi}},\ \ z\in \C^n.$$
    Then $\actf{z}{\varphi_\xi}\in \LtwoF$, for all $z\in \C^n$ and hence the Toeplitz operator $T_{\varphi_\xi}$ is densely defined on $\dom$.
    It is known that $T_{\varphi_\xi}$ is bounded if and only if $|1+\frac{1}{\xi}|\geq 1$.
    This is because $T_{\varphi_\xi}$ is diagonalizable w.r.t. the monomial basis $\{e_\vm\}$ with eigenvalues given by
    \begin{align*}
        \ip{T_{\varphi_\xi} e_\vm}{e_\vm}= \frac{1}{\xi^n}(1+1/\xi)^{-(|\vm|+n)}.
    \end{align*}

    \end{example}

\subsection{The convolution of two operators}

Let $S$ and $T$ be two densely defined operators s.t. $S\actop{z}{T}\in \traceop$ for all $z\in \C^n$. Then we define the convolution $S\ast T:\C^n\to \C$ by
$$(S\ast T)(z)=\Tr(S\actop{z}{T}),\ \ z\in \C^n.$$

Notice that the Berezin transform of a densely defined operator $S$ on $\dom$ can still be defined by
$$B(S)(z)=\ip{Sk_z}{k_z},\ \ z\in \C^n.$$
It follows by a simple computation that
$B(S)=S\ast \Phi.$ Moreover, if $a\in \symb$, we have $B(T_a)=a\ast \varphi_1$.

The Berezin transform remains injective even on the class of densely defined operators with domain $\dom$. This is a particularly useful property which we will use below to show that two densely defined operators are equal. The proof of it is essentially the same as for the bounded case. 

\begin{proposition}\label{prop:Berezin injectivity}\cite[Proposition 3.1]{Z12}
    Let $S$ and $T$ be two densely defined operators on $\dom$.
    If the Berezin transforms of $S$ and $T$ coincide, i.e.  $S\ast \Phi=T\ast \Phi$, then $S=T$ on $\dom$.
\end{proposition}


\section{Berger-Coburn theorem and related results}

In this section, we present our main results, a short proof of the Berger-Coburn theorem and a quantitative analog for Schatten class membership of Toeplitz operators. Additionally, we revisit compactness of densely defined Toeplitz operators.  

\subsection{The heat-flow of a Toeplitz operator}
The operators $\{\Phi_t\}_{t>0}$ were first introduced by Berger and Coburn \cite{BC94} with a slightly different parametrization. Berger and Coburn used it to discuss the heat flow of a Toeplitz operator. We make the following observation.

\begin{proposition}\label{prop:heat_flow}
    Let $t>\frac{1}{2}$ and let $a\in \symb$. If $T_a\in \bdd$, then $$B_t(a)=T_a\ast \Phi_{t-1/2}.$$
\end{proposition}

\begin{proof}
    As usual, it is enough to show that the Berezin transforms of two functions coincide. We have
    \begin{align*}
        (T_a\ast \Phi_{t-1/2})\ast \varphi_1 &= (T_a\ast \Phi_{t-1/2})\ast (\Phi_{1/2}\ast \Phi_{1/2})\\
        &=(T_a\ast \Phi_{1/2})\ast (\Phi_{t-1/2}\ast \Phi_{1/2})\\
        &=(a\ast \varphi_1)\ast \varphi_t\\
        &=a\ast \varphi_{t+1}\\
        &=(a\ast \varphi_t)\ast \varphi_1.
    \end{align*}   
Note that all of the operators appearing above are bounded, so here we only use the associativity and commutativity of standard operator convolutions, and  Proposition \ref{prop:function_convol}.

\end{proof}


Next we prove the Berger-Coburn theorem. Note that in the second part of the theorem, in addition to the classical statement, we provide an explicit formula for the bounded extension of the densely defined Topelitz operator $T_a$.  

\begin{theorem} \label{theo:BC}
Let $a\in \symb$. If $T_a$ is a bounded operator then $B_t(a)$ is a bounded function for all $t>1/2$, and
\[\norm{B_t(a)}_{\infty}\leq \frac{1}{(2 t-1)^n}\norm{T_a}.\]
On the other hand, if $B_t(a)$ is a bounded function for some $0<t<1/2$, then $T_a$ is a bounded operator in $\toepalg$ and 
$$T_a=B_t(a)\ast \Phi_{1/2-t}.$$
In addition,
\[\norm{T_a}\leq \frac{1}{(1-2 t)^n}\norm{B_t(a)}_{\infty}.\]
\end{theorem}

\begin{proof}
    To prove the first claim, note that by Proposition \ref{prop:heat_flow},
    \[\norm{B_t(T_a)}_{\infty}=\norm{T_a\ast \Phi_{t-1/2}}_\infty\leq \norm{\Phi_{t-1/2}}_1\norm{T_a}=\frac{1}{(2 t-1)^n}\norm{T_a}.\]

To prove the second statement, we assume $t\in (0,1/2)$. Then the operator $B_t(a)\ast \Phi_{1/2-t}$ is a bounded operator in $\toepalg$ as $B_t(a) \in \bddf$ and $\Phi_{1/2-t}\in \traceop$.
Now we prove
$$T_a=B_t(a)\ast \Phi_{1/2-t}$$
on $\dom$, i.e. the Toeplitz operator $T_a$ has the bounded extension $B_t(a)\ast \Phi_{1/2-t}$.

Due to Proposition \ref{prop:Berezin injectivity}, one only has to check if the Berezin transforms of the two operators coincide. Note that
$$T_a\ast \Phi_{1/2}=a\ast \varphi_1.$$
On the other hand, by statement (3) in Proposition \ref{prop:function_convol},
$$(B_t(a)\ast \Phi_{1/2-t})\ast \Phi_{1/2}=(a\ast \varphi_{t})\ast (\Phi_{1/2-t}\ast \Phi_{1/2})=(a\ast \varphi_{t})\ast \varphi_{1-t}=\varphi_{1}\ast a$$
proving the claim.
Then $T_a\in \toepalg$ and by the QHA Young's inequality,
\[\norm{T_a}\leq \norm{B_t(a)}_\infty\norm{\Phi_{1/2-t}}_1=\frac{1}{(1-2 t)^n}\norm{B_t(a)}_\infty.\]
\end{proof}

\subsection{Schatten classes}
In \cite{BCI10},  Bauer, Coburn, and Isralowitz found a sufficient condition for the Schatten class membership of densely defined Toeplitz operators using the heat flow. Conversely, in \cite{F20}, Fulsche provides a necessary condition for the same. Here, we prove the following quantitative version of their collective theorem for Schatten class membership, answering one of the questions raised in \cite{BCI10}.

\begin{theorem} \label{theo:Schatten}
Let $1\leq p<\infty$ and let $a\in \symb$. If $T_a\in \schop$ then $B_t(a)\in \Lpf$ for all $t>1/2$, and
\[\norm{B_t(a)}_{p}\leq \frac{1}{(2 t-1)^n}\norm{T_a}_p.\]
On the other hand, if $B_t(a)\in \Lpf$ for some $0<t<1/2$, then $T_a\in \schop$ and 
$$T_a=B_t(a)\ast \Phi_{1/2-t}.$$
In addition,
\[\norm{T_a}_p\leq \frac{1}{(1-2 t)^n}\norm{B_t(a)}_{p}.\]
\end{theorem}

\begin{proof}
    Note that by Proposition \ref{prop:heat_flow},
    \[\norm{B_t(T_a)}_{p}=\norm{T_a\ast \Phi_{t-1/2}}_p\leq  \norm{\Phi_{t-1/2}}_1\norm{T_a}_p=\frac{1}{(2 t-1)^n}\norm{T_a}_p,\]
    for $t>\frac{1}{2}$.

For the other implication, assume $B_t(a)\in \Lpf$ for some $t\in(0,\frac{1}{2})$. Then,  $B_t(a)\ast \Phi_{1/2-t}\in \schop$. The equality
$$T_a=B_t(a)\ast \Phi_{1/2-t}$$
is proved as before by comparing the Berezin transforms.
Finally, the desired inequality then follows from the QHA Young's inequality,

\[\norm{T_a}_p=\norm{B_t(a)\ast \Phi_{t-1/2}}_p\leq \norm{B_t(a)}_p\norm{\Phi_{t-1/2}}_1\leq \frac{1}{(1-2 t)^n}\norm{B_t(a)}_{p}.\]
\end{proof}

\subsection{Compactness}

Let $\czero$ denote the algebra of functions on $\C^n$ that vanish at infinity and let  $\comp$ denote the algebra of all compact operators on $\F$. It is well-known that if $a\in \czero$ then $T_a$ is compact.  Also, if $S\in \comp$, then $B(S)\in \czero$. The latter holds more generally: if $S\in \comp$ and $T\in \traceop$, then $S\ast T\in \czero$. Indeed, this is easy to see when $S$ is rank-one. The general case then follows by standard linearity and density arguments. The following theorem, which characterizes the compactness of densely defined Toeplitz operators, was proved in \cite{BCI10}. Here, we present a straightforward proof that avoids the use of pseudo-differential operators.

\begin{theorem}
    Let $a\in \symb$. If $T_a\in \comp$ then $B_t(a)\in \czero$ for all $t>\frac{1}{2}$. If $B_t(a)\in \czero$ for some $t\in (0,\frac{1}{2})$, then $T_a$ is compact.
\end{theorem}

\begin{proof}
    Assume $t>\frac{1}{2}$. Then, using that $T_a$ is compact and $\Phi_{t-1/2}$ is trace class we obtain 
    $$B_t(a)=T_a\ast \Phi_{t-1/2}\in \czero.$$
    
    For the other implication, assume $B_t(a)\in \czero$ for some $t\in (0,\frac{1}{2})$. Let $\epsilon>0$.
    Then since $\czero=\overline{\czero\ast \varphi_t}$, there exists $b\in \czero$ s.t. 
    $$\|B_t(a)-B_t(b)\|_\infty <\epsilon.$$
    Therefore, 
    \begin{align*}
        \|T_a-T_b\| &= \|T_{a-b}\|\\
        &=\|B_t(a-b)\ast \Phi_{1/2-t}\|\\
        &\leq \frac{1}{(1-2 t)^n}\|B_t(a-b)\|\\
        &\leq \frac{1}{(1-2 t)^n}\|B_t(a)-B_t(b)\|\\
        &< \frac{1}{(1-2 t)^n} \epsilon.
    \end{align*}
    Also, since $b\in \czero$, we have $T_b\in \comp$.
    Thus, $T_a\in \overline{\comp}=\comp$.
\end{proof}

\noindent \textbf{Acknowledgments :} We thank Robert Fulsche for useful comments and suggestions.


\end{document}